\documentclass[12pt]{amsart}
\usepackage[top=1.6cm,bottom=2cm,left=1.1in,right=1.1in]{geometry}
\usepackage[T1]{fontenc}
\usepackage[utf8]{inputenc}
\usepackage{lmodern}
\usepackage{amsmath,amssymb,amsthm,mathtools}
\usepackage{mathrsfs}
\usepackage{enumitem}
\usepackage{microtype}
\usepackage[colorlinks=true,linkcolor=blue,citecolor=blue,urlcolor=blue]{hyperref}

\numberwithin{equation}{section}

\newtheorem{theorem}{Theorem}[section]
\newtheorem{lemma}[theorem]{Lemma}
\newtheorem{corollary}[theorem]{Corollary}
\newtheorem{proposition}[theorem]{Proposition}
\newtheorem{conjecture}[theorem]{Conjecture}
\theoremstyle{definition}
\newtheorem{definition}[theorem]{Definition}
\theoremstyle{remark}
\newtheorem{remark}[theorem]{Remark}

\newcommand{\ii}{\sqrt{-1}}

\newcommand{\BC}{\mathrm{BC}}
\newcommand{\A}{\mathrm{A}}
\newcommand{\MK}{\mathcal{MK}}
\newcommand{\MN}{\mathcal{MN}}
\newcommand{\G}{\mathcal{G}}
\newcommand{\B}{\mathcal{B}}
\newcommand{\E}{\mathcal{E}}
\newcommand{\EA}{\mathcal{E}_{\A}}
\newcommand{\Rea}{\operatorname{Re}}
\newcommand{\Ima}{\operatorname{Im}}
\newcommand{\arccot}{\operatorname{arccot}}

\title[A numerical criterion for the 2-Hessian equation]{A Numerical Criterion for the 2-Hessian Equation on Compact K\"ahler Manifolds}

\author{Jixiang Fu}
\address{Shanghai Center for Mathematical Sciences,
	Fudan University,
	Shanghai 200433, China}
\email{majxfu@fudan.edu.cn}

\author{Dekai Zhang}
\address{School of Mathematical Sciences, Key Laboratory of Mathematics and Engineering Applications (Ministry of Education), Shanghai Key Laboratory of PMMP, East China Normal University, Shanghai 200241, China. 
}
\email{dkzhang@math.ecnu.edu.cn}
\author{Ziyi Zhang}
\address{
School of Mathematical Sciences, Fudan University, Shanghai 200433, China}
\email{21210180101@m.fudan.edu.cn}

\date{}

\begin{document}
	
	\begin{abstract}
		We show that a Nakai--Moishezon-type criterion associated with the complex $2$-Hessian equation produces a Gauduchon class. In complex dimension three, this numerical criterion is equivalent to the existence of a smooth $2$-admissible representative and hence to the solvability of the $2$-Hessian equation. As consequences of these results, we prove the corresponding conjectures of Murakami for the complex Hessian equation and of Sz\'ekelyhidi for the Hessian quotient equation in dimension three. We also establish a boundary version of the above results.
	\end{abstract}
	
	\maketitle
	
	\section{Introduction}
	
	Let $(M,\omega)$ be an $n$-dimensional compact K\"ahler manifold, and let $\chi$ be a real closed $(1,1)$-form. We denote by $\Gamma^k_\omega$ the cone of smooth real $(1,1)$-forms $\alpha$ satisfying
	$$
	\alpha^j\wedge\omega^{n-j}>0, 1\leq j\leq k.
	$$ If $\alpha\in \Gamma_{\omega}^k$, we say $\alpha$ is a $k$-admissible form.
	For a smooth function $u$, write $\chi_u=\chi +\ii\partial\bar\partial u$. Consider the following equations.
	\begin{enumerate}[label=(\roman*),leftmargin=2.4em]
		\item The complex $k$-Hessian equation, for $2\leq k\leq n$,
		\begin{equation}\label{eq:complex-hessian}
			\chi_u^k\wedge\omega^{n-k}=f\omega^n,\qquad \chi_u\in\Gamma^k_\omega,
		\end{equation}
		where $f$ is a smooth positive function satisfying $\int_{M}f\omega^n=\int_{M}\chi^k\wedge \omega^{n-k}$.
		\item The complex $(k,l)$-Hessian quotient equation, for $1\leq l<k\leq n$,
		\begin{equation}\label{eq:hessian-quotient}
			\chi_u^k\wedge\omega^{n-k}=c\chi_u^l\wedge\omega^{n-l},\qquad \chi_u\in\Gamma^k_\omega,
		\end{equation}
		where $c$ is a positive constant satisfying ${\int_{M}\chi^{k}\wedge\omega^{n-k}}=c{\int_{M}\chi^{l}\wedge\omega^{n-l}}.$ 
	\end{enumerate}
    
	The case $k=n$ of equation~\eqref{eq:complex-hessian} is the complex
	Monge--Amp\`ere equation solved by Yau~\cite{Yau}; see
	Tosatti--Weinkove~\cite{TosattiWeinkove} for the Hermitian case.
	When $\chi=\omega$, Hou--Ma--Wu~\cite{HouMaWu} established the
	second-order estimate, and Dinew--Ko\l odziej~
	\cite{DinewKolodziej} subsequently obtained the gradient estimate,
	thereby completing the proof of smooth solvability on compact
	K\"ahler manifolds. The case of a general closed $(1,1)$-form $\chi\in \Gamma^{k}_{\omega}$ for the $k$-Hessian equation \eqref{eq:complex-hessian}  was  solved by Sz\'ekelyhidi \cite{Szekelyhidi}. 
    
    When \(k = n\), \(l = n-1\), equation~\eqref{eq:hessian-quotient} is the \(J\)-equation introduced by Donaldson~\cite{Donaldson1999} and Chen~\cite{Chen2000}, and it  was solved by Song--Weinkove \cite{SongWeinkove} under the  subsolution condition; the case \(k = n, 1 \leq l \leq n-1\) was solved by Fang--Lai--Ma  \cite{FangLaiMa}. The general case was subsequently solved by  Sz\'{e}kelyhidi \cite{Szekelyhidi}.
	
	\begin{theorem}[Sz\'ekelyhidi~\cite{Szekelyhidi}]\label{thm:szekelyhidi-subsolution}
		Let $(M,\omega)$ be a compact K\"ahler manifold of dimension $n$. Assume  $\chi\in\Gamma^k_\omega$ is closed and satisfies
		\begin{equation}\label{eq:subsolutioncondition}
		k\chi^{k-1}\wedge\omega^{n-k}-cl\chi^{l-1}\wedge\omega^{n-l}>0.
		\end{equation} Then there exists a smooth function $u$ solving  equation~\eqref{eq:hessian-quotient}.
	\end{theorem}

Sz\'ekelyhidi conjectured that the subsolution condition~\eqref{eq:subsolutioncondition} in Theorem~\ref{thm:szekelyhidi-subsolution}  is equivalent to numerical inequalities on analytic subvarieties.

\begin{conjecture}[Sz\'ekelyhidi~\cite{Szekelyhidi}]\label{conj:szekelyhidi}
	Let $(M,\omega)$ be a compact K\"ahler manifold of dimension $n$. Assume  $\chi\in\Gamma^k_\omega$ is closed. Then there exists a smooth function $u$ such that $$\chi_u\in\Gamma_{\omega}^k, \qquad k\chi_u^{k-1}\wedge\omega^{n-k}-cl\chi_u^{l-1}\wedge\omega^{n-l}>0$$ if and only if, for any $p$-dimensional irreducible analytic subvariety $V\subset M$ with $n-l\leq p\leq n-1$,
	\[
	\frac{k!}{(p-n+k)!}\int_V\chi^{p-n+k}\wedge\omega^{n-k}
	-\frac{l!}{(p-n+l)!}c\int_V\chi^{p-n+l}\wedge\omega^{n-l}>0.
	\]
\end{conjecture}

The $J$-equation case was solved by Chen \cite{Chen} in the uniform version and finally by Song \cite{Song}; the case $k=n$ and $1\leq l\leq n-1$ was solved by Datar--Pingali~\cite{DatarPingali} on projective manifolds and finally by Fang--Ma
~\cite{FangMa}.
Motivated by Conjecture~\ref{conj:szekelyhidi}, Murakami proposed the following numerical conjecture of the existence of a $k$-admissible representative.

\begin{conjecture}[Murakami~\cite{Murakami}]\label{conj:murakami}
Let $(M,\omega)$ be a compact K\"ahler manifold of dimension $n\geq2$, and let $\chi$ be a real closed $(1,1)$-form. Then $[\chi]$ contains a smooth representative in $\Gamma^k_\omega$ if and only if, for any $t\geq0$ and any $p$-dimensional irreducible analytic subvariety $V\subset M$ with $n-k+1\leq p\leq n$,
\[
\int_V(\chi+t\omega)^{p-n+k}\wedge\omega^{n-k}>0.
\]
\end{conjecture}

For $k=n$, this follows from the numerical characterization of the K\"ahler cone by Demailly--P\u{a}un~\cite{DemaillyPaun}. Murakami~\cite{Murakami} proved the conjecture on the Calabi-symmetric manifolds and for semiample classes. In this paper, we restrict  to the case $k=2$.

Our first main result is a numerical criterion ensuring that the Aeppli class of $\chi\wedge\omega^{n-2}$ lies in the Gauduchon cone.

\begin{theorem}\label{lem:gauduchon-criterion}
Let $(M,\omega)$ be a compact K\"ahler manifold of dimension $n\geq2$, and let $\chi$ be a real closed $(1,1)$-form. Suppose that
\[
\int_M\chi^2\wedge\omega^{n-2}>0,\qquad \int_M\chi\wedge\omega^{n-1}>0,
\]
and that, for any irreducible divisor $D\subset M$,
\[
\int_D\chi\wedge\omega^{n-2}>0.
\]
Then
\(
[\chi\wedge\omega^{n-2}]_{\A}\in H^{n-1,n-1}_{\A}(M,\mathbb{R})
\)
belongs to the Gauduchon cone.
\end{theorem}

The proof combines the duality between the pseudo-effective and Gauduchon cones with Boucksom's divisorial Zariski decomposition and the Hodge index theorem. 

As an application of Theorem~\ref{lem:gauduchon-criterion}, based on Chen's  theorem~\cite{Chen} for the twisted supercritical  deformed Hermitian--Yang--Mills equation, we can solve the $2$-Hessian equation in dimension three under a Nakai--Moishezon-type numerical condition.

\begin{theorem}\label{thm:main-threefold}
Let $(M,\omega)$ be a compact K\"ahler threefold, and let $\chi$ be a real closed $(1,1)$-form satisfying 
\begin{equation}\label{eq:normalization}
\int_M\chi^2\wedge\omega>0.
\end{equation}
The following three statements are equivalent.
\begin{enumerate}[label=(\arabic*),leftmargin=2.4em]
\item $
\int_M\chi\wedge\omega^2>0
$
and, for any irreducible analytic surface $Y\subset M$,
\[
\int_Y\chi\wedge\omega>0.
\]
\item There exists a smooth $(1,2)$-form $\psi$ such that
\[
\chi\wedge\omega+\partial\psi+\bar\partial\bar\psi>0
\]
as a smooth $(2,2)$-form.
\item For any smooth positive function $f$ satisfying
$
\int_M\chi^2\wedge\omega=\int_M f\omega^3,
$
there exists a smooth function $u$ solving
\[
\chi_u^2\wedge\omega=f\omega^3,\quad \chi_u\in\Gamma^2_\omega.
\]
\end{enumerate}
\end{theorem}
 
As consequences of Theorem~\ref{thm:main-threefold}, we  prove Conjectures~\ref{conj:szekelyhidi} and~\ref{conj:murakami} when $n=3$ and $k=2$. Together with the previously known cases for $k=n$, our results fill the remaining case in complex dimension three and
thus show that Conjectures~\ref{conj:szekelyhidi} and~\ref{conj:murakami} hold on compact K\"ahler threefolds.

Finally, we consider the boundary case. The following statement is the boundary version of Theorem~\ref{lem:gauduchon-criterion}.

\begin{theorem}\label{prop:boundary-gauduchon}
Let $(M,\omega)$ be a compact K\"ahler manifold of dimension $n$, and let $\chi$ be a real closed $(1,1)$-form. Suppose that
\[
\int_M\chi^2\wedge\omega^{n-2}>0,\qquad \int_M\chi\wedge\omega^{n-1}>0,
\]
and that, for any irreducible divisor $D\subset M$,
\[
\int_D\chi\wedge\omega^{n-2}\geq0.
\]
Then there exist finitely many exceptional prime divisors $D_1,\ldots,D_m$ and positive numbers $a_1,\ldots,a_m$ such that
\[
\left([\chi]-\sum_{i=1}^ma_i\{D_i\}\right)\wedge[\omega]^{n-2}
\]
belongs to the Gauduchon cone and
\[
\int_M\left([\chi]-\sum_{i=1}^ma_i\{D_i\}\right)^2\wedge[\omega]^{n-2}>0.
\]
\end{theorem}

As an application of Theorem~\ref{prop:boundary-gauduchon}, Theorem~\ref{thm:main-threefold}, and the pluripotential theory developed in~\cite{PangSunWangZhou}, we obtain the following result. Here $\overline{\Gamma}^2_{\omega}$ denotes the closure of the cone $\Gamma_{\omega}^2$.

\begin{corollary}\label{thm:boundary-hessian}
Let $(M,\omega)$ be a compact K\"ahler threefold, and let $\chi$ be a real closed $(1,1)$-form. Suppose that
\[
\int_M\chi^2\wedge\omega>0,\qquad \int_M\chi\wedge\omega^2>0,
\]
and that, for any irreducible analytic surface $Y\subset M$,
\[
\int_Y\chi\wedge\omega\geq0.
\]
Then $\chi$ is $(\omega,2)$-big. In particular, suppose $\chi\in\overline{\Gamma}^2_{\omega}$ and
$
\int_M\chi^2\wedge\omega>0,
$ and let $0 
\leq f \in L^p(M,\omega^3)$ with $p>\frac{3}{2}$ and
$\int_M f\,\omega^3>0$. 
Then for any constant $\lambda>0$,  there exists a unique  function
$
u\in\operatorname{SH}_2(M,\chi,\omega)\cap L^\infty(M)$
satisfying
\[
(\chi+\ii\partial\bar\partial u)^2\wedge\omega=e^{\lambda u}f\omega^3
\]
in the potential sense. Furthermore, we have
$
\operatorname{osc}_M u \leq C,$
where $C$ is a constant depending on
$\chi,\omega,\lambda,p,M,\|f\|_p$.
\end{corollary}

The paper is organized as follows. Section~2 recalls cone duality, Boucksom's divisorial Zariski decomposition, and the relevant intersection-theoretic facts. Section~3 proves Theorem~\ref{lem:gauduchon-criterion}. Section~4 proves Theorem~\ref{thm:main-threefold}. Section~5 derives the two numerical conjectures in dimension three. Section~6 proves Theorem~\ref{prop:boundary-gauduchon} and Corollary~\ref{thm:boundary-hessian}.

\section{Preliminaries}

\subsection{Bott--Chern and Aeppli cones}

Let $X$ be a compact complex manifold. The Bott–Chern cohomology group  is defined as 
\[
H^{p,p}_{\BC}(X,\mathbb{R})=
\frac{\{\alpha\in A^{p,p}_{\mathbb{R}}(X):d\alpha=0\}}
{\ii\partial\bar\partial A^{p-1,p-1}_{\mathbb{R}}(X)},
\]
and the Aeppli cohomology group is defined as 
\[
H^{p,p}_{\A}(X,\mathbb{R})=
\frac{\{\alpha\in A^{p,p}_{\mathbb{R}}(X):\partial\bar\partial\alpha=0\}}
{\bigl(\partial A^{p-1,p}(X)+\bar\partial A^{p,p-1}(X)\bigr)\cap A^{p,p}_{\mathbb{R}}(X)}.
\]
These groups can equivalently be defined using currents. We consider the following cones.

\begin{definition}\label{def:cones}
Let X be a compact complex manifold of  dimension $n$.
\begin{enumerate}[label=(\arabic*),leftmargin=2.4em]
\item The pseudo-effective cone $\E\subset H^{1,1}_{\BC}(X,\mathbb{R})$ is generated by $d$-closed positive $(1,1)$-currents.
\item The Aeppli pseudo-effective cone $\EA\subset H^{1,1}_{\A}(X,\mathbb{R})$ is generated by positive $\partial\bar\partial$-closed $(1,1)$-currents.
\item The balanced cone $\B\subset H^{n-1,n-1}_{\BC}(X,\mathbb{R})$ consists of classes containing smooth, $d$-closed, strictly positive $(n-1,n-1)$-forms.
\item The Gauduchon cone $\G\subset H^{n-1,n-1}_{\A}(X,\mathbb{R})$ consists of classes containing smooth, $\partial\bar\partial$-closed, strictly positive $(n-1,n-1)$-forms.
\end{enumerate}
\end{definition}

The natural integration pairings between the Bott–Chern cohomology group and  the Aeppli cohomology group identify the relevant dual cones. We recall the following  duality statements.

\begin{lemma}[Fu--Xiao~\cite{FuXiao}]\label{lem:fu-xiao-duality}
Let $X$ be a compact complex manifold, and let $\Omega$ be a real $d$-closed $(n-1,n-1)$-form. If
\[
\int_X\Omega\wedge T>0
\]
for any nonzero positive $\partial\bar\partial$-closed $(1,1)$-current $T$, then $[\Omega]_{\BC}\in\B$.
\end{lemma}

\begin{lemma}[Xiao~\cite{XiaoVolume,XiaoInverse}]\label{lem:xiao-duality}
Let $X$ be a compact complex manifold, and let $\Omega$ be a real $\partial\bar\partial$-closed $(n-1,n-1)$-form. If
\[
\int_X\Omega\wedge T>0
\]
for any nonzero positive $d$-closed $(1,1)$-current $T$, then $[\Omega]_{\A}\in\G$.
\end{lemma}

Lemma~\ref{lem:xiao-duality} is one form of the duality $\overline{\G}=\E^\vee$ needed in this paper.

\subsection{Divisorial Zariski decomposition}

We recall the terminology of Boucksom~\cite{Boucksom}. Fix a Hermitian form $\omega_0$ on $X$.

\begin{definition}\label{def:modified-cones}
A class $\alpha\in H^{1,1}_{\BC}(X,\mathbb{R})$ is \emph{modified K\"ahler} if it contains a K\"ahler current $T$ whose generic Lelong number along any prime divisor is zero. It is \emph{modified nef} if, for any $\varepsilon>0$, it contains a closed current $T_\varepsilon\geq-\varepsilon\omega_0$ whose generic Lelong number along any prime divisor is zero.
\end{definition}

The modified K\"ahler and modified nef cones are denoted by $\MK$ and $\MN$, respectively. The latter is the closure of the former.

\begin{proposition}[Boucksom~\cite{Boucksom}]\label{prop:modified-kahler}
A class $\alpha$ lies in $\MK$ if and only if there exist a modification $\mu:\widetilde X\to X$ and a K\"ahler class $\widetilde\alpha$ on $\widetilde X$ such that $\alpha=\mu_*\widetilde\alpha$.
\end{proposition}

Let $\alpha$ be a pseudo-effective class. Its minimal multiplicity at $x\in X$ is denoted by $\nu(\alpha,x)$, and its generic minimal multiplicity along a prime divisor $D$ is
\[
\nu(\alpha,D)=\inf_{x\in D}\nu(\alpha,x).
\]
The series $\sum_D\nu(\alpha,D)[D]$ converges as a positive current.

\begin{definition}\label{def:zariski-decomposition}
The negative part and the Zariski projection of $\alpha$ are
\[
N(\alpha)=\sum_D\nu(\alpha,D)D,\qquad Z(\alpha)=\alpha-\{N(\alpha)\}.
\]
The decomposition $\alpha=Z(\alpha)+\{N(\alpha)\}$ is called the divisorial Zariski decomposition.
\end{definition}

\begin{proposition}[Boucksom~\cite{Boucksom}]\label{prop:zariski-modified-nef}
For any pseudo-effective class $\alpha$, the class $Z(\alpha)$ is modified nef.
\end{proposition}

A finite family of prime divisors is called \emph{exceptional} if the convex cone generated by their classes meets $\MN$ only at the origin. Boucksom proved that the classes of an exceptional family are linearly independent; see Proposition~3.11(3) of~\cite{Boucksom}. Consequently, the cardinality of such a family is bounded by the Picard number.

\subsection{The intersection form}

Let $(M,\omega)$ be a compact K\"ahler manifold of dimension $n$. Define the intersection form 
\begin{equation}\label{eq:intersection-form}
q(\alpha,\beta)=\int_M\alpha\wedge\beta\wedge\omega^{n-2},\qquad \alpha,\beta\in H^{1,1}(M,\mathbb{R}).
\end{equation}

\begin{lemma}[Hodge index inequality, see Section 6.3.2 of \cite{Voisin}]\label{lem:hodge-index-inequality}
For every $\alpha\in H^{1,1}(M,\mathbb{R})$,
\begin{equation}\label{eq:hodge-index-inequality}
\left(\int_M\alpha^2\wedge\omega^{n-2}\right)\left(\int_M\omega^n\right)
\leq\left(\int_M\alpha\wedge\omega^{n-1}\right)^2.
\end{equation}
Equality holds if and only if $\alpha$ is a real multiple of $[\omega]$.
\end{lemma}

The Hodge index theorem~\cite{Voisin} states that $q$ has signature $(1,h^{1,1}(M)-1)$. In particular, if $q(\alpha,\alpha)>0$, then $q$ is negative definite on $\alpha^\perp$.

 {\begin{lemma}\label{lem:modified-nef-square}
			If $\gamma$ is a modified nef class, then
			\[
			q(\gamma,\gamma)\geq 0.
			\]
		\end{lemma}
		
		\begin{proof}
			The argument is similar to Boucksom's proof of Proposition~4.2(i) in~\cite{Boucksom}. We first assume that $\gamma$ is a modified K\"ahler class.
			By Proposition~\ref{prop:modified-kahler}, there exist a modification $\mu:\widetilde M\to M$ and a K\"ahler class $\widetilde\gamma$ on $\widetilde M$ such that $\gamma=\mu_*\widetilde\gamma$. Let $\widetilde\beta$ be a K\"ahler form in the class $\widetilde\gamma$. Then $\mu_*\widetilde\beta$ is a current in the class $\gamma$ with analytic singularities along a subvariety of codimension at least two.
			By the intersection theory of currents 
			 \cite[Section~2.6]{Boucksom}, the product
			$\mu_*\widetilde\beta\wedge\mu_*\widetilde\beta$
			is well defined as a closed positive $(2,2)$-current and belongs to
			the cohomology class  $\gamma^2$. Since $\omega^{n-2}$ is a smooth positive closed form of bidegree $(n-2,n-2)$, we have
			\[
			\begin{aligned}
				q(\gamma,\gamma)
				&=\int_M\gamma^2\wedge\omega^{n-2}
=\int_M \mu_*\widetilde\beta\wedge\mu_*\widetilde\beta\wedge\omega^{n-2}\geq0 .
			\end{aligned}
			\]
			
			Now suppose that  $\gamma$  is a modified nef class.  Then for any $\varepsilon>0$,
			$
			\gamma_{\varepsilon}=\gamma+\varepsilon[\omega]
			$ is a modified K\"ahler class. By the first part,
			$
			q(\gamma_{\varepsilon},\gamma_{\varepsilon})\geq0 .$
			Passing to the limit and using the continuity of the intersection form
			$q$ gives
			$
			q(\gamma,\gamma)\geq0 .
			$
\end{proof}}

By abuse of notation, for smooth real closed $(1,1)$-forms $\chi_1$, $\chi_2$,  we usually write $q(\chi_1,\chi_2)$ to denote $q([\chi_1],[\chi_2])$ when no confusion can arise. 

\section{A numerical criterion for a Gauduchon class}

We now prove the numerical criterion stated in the introduction.

\begin{theorem}\label{thm:gauduchon-class}
Let $(M,\omega)$ be a compact K\"ahler manifold of dimension $n\geq2$, and $\chi$ be a real closed $(1,1)$-form. Let $q$ be the intersection form on $H^{1,1}(M,\mathbb{R})$ defined by~\eqref{eq:intersection-form}. Assume that
\[
q(\chi,\chi)>0,\qquad q(\chi,\omega)>0,
\]
and
\[
q(\chi,\{D\})>0
\]
for any prime divisor $D$. Then $[\chi\wedge\omega^{n-2}]_{\A}\in\G$.
\end{theorem}

\begin{proof}
The proof is motivated by~\cite{XiaoRemark,XiaoInverse}. By Lemma~\ref{lem:xiao-duality}, it suffices to prove that
\begin{equation}\label{eq:positive-pairing-current}
\int_M\chi\wedge\omega^{n-2}\wedge T>0
\end{equation}
for any nonzero positive closed $(1,1)$-current $T$. Write the divisorial Zariski decomposition of its class as
\[
\{T\}=\gamma+\sum_D\nu(\{T\},D)\{D\},\qquad \gamma=Z(\{T\})\in\MN.
\]
If $\gamma=0$, then at least one coefficient $\nu(\{T\},D)$ is positive, and the assumptions on divisors immediately give~\eqref{eq:positive-pairing-current}. Suppose now that $\gamma\neq0$. Since $\gamma$ is a nonzero pseudo-effective class,
\[
q(\omega,\gamma)=\int_M\gamma\wedge\omega^{n-1}>0.
\]
We claim that $q(\chi,\gamma)>0$. Otherwise, set
\[
c=-\frac{q(\chi,\gamma)}{q(\omega,\gamma)}\geq0,\qquad \alpha=\chi+c\omega.
\]
Then $q(\alpha,\gamma)=0$, while
\[
q(\alpha,\alpha)=q(\chi,\chi)+2cq(\chi,\omega)+c^2q(\omega,\omega)>0.
\]
The Hodge index theorem therefore gives $q(\gamma,\gamma)<0$, because $\gamma\neq0$ lies in $\alpha^\perp$. This contradicts Lemma~\ref{lem:modified-nef-square}. Hence $q(\chi,\gamma)>0$. Finally,
\[
\int_M\chi\wedge\omega^{n-2}\wedge T
=q(\chi,\gamma)+\sum_D\nu(\{T\},D)q(\chi,\{D\})>0,
\]
which proves the theorem.
\end{proof}

\begin{proof}[Proof of Theorem~\ref{lem:gauduchon-criterion}]
The assumptions of Theorem~\ref{lem:gauduchon-criterion} are precisely those of Theorem~\ref{thm:gauduchon-class}, in view of the definition of $q$. 
\end{proof}

\begin{remark}\label{rem:positive-square-essential}
The condition $q(\chi,\chi)>0$ is essential for this argument. Without it, positivity against divisors and against $[\omega]$ does not, in general, ensure that the Aeppli class $[\chi\wedge\omega^{n-2}]_{\A}$ is Gauduchon.
\end{remark}

\section{The 2-Hessian equation in dimension three}

We first record an elementary fact used in the proof.
 
For a Hermitian form $\eta$  and a real $(1,1)$-form $\alpha$ on  a compact complex threefold, let $\lambda_1,\lambda_2,\lambda_3$ be the eigenvalues of $\alpha$ with respect to $\eta$, and set $\theta_j=\arccot\lambda_j\in(0,\pi)$. Define
\[
P(\alpha)=\max_{1\leq i\leq3}\sum_{j\neq i}\theta_j,
\qquad
Q(\alpha)=\sum_{j=1}^3\theta_{j}.
\]
For constants $0<\theta<\Theta<\pi$, the phase cone $\Gamma_{\eta,\theta,\Theta}$ consists of the forms satisfying $P(\alpha)<\theta$ and $Q(\alpha)<\Theta$.

\begin{lemma}[Yuan \cite{Yuan}]\label{lem:phase-cone}
If
$
\theta_1+\theta_2+\theta_3<\pi$, then $\alpha\in\Gamma_{\eta}^2$. In particular,  $\Gamma_{\eta,\theta,\Theta} \subset \Gamma^2_\eta$.
\end{lemma}

Now we prove Theorem ~\ref{thm:main-threefold}.
\begin{proof}[Proof of Theorem~\ref{thm:main-threefold}]
Assume first that statement~(3) holds. Since $\chi_u\in\Gamma^2_\omega$, the linearization of $\sigma_2$ is positive, which in dimension three is equivalent to the positivity of $\chi_u\wedge\omega$ as a $(2,2)$-form. Hence
\[
\int_M\chi\wedge\omega^2>0,
\qquad
\int_Y\chi\wedge\omega>0
\]
for any irreducible surface $Y$. Thus statement~(3) implies statement~(1).

Assume statement~(1). Then Theorem~\ref{thm:gauduchon-class} shows that $[\chi\wedge\omega]_{\A}$ lies in the Gauduchon cone. Hence there exists a smooth $(1,2)$-form $\psi$ such that
\[
\chi\wedge\omega+\partial\psi+\bar\partial\bar\psi>0,
\]
which is statement~(2).

It remains to prove that statement~(2) implies statement~(3). By compactness, there exists a constant $\delta>0$ such that
\begin{equation}\label{eq:gauduchon-lower-bound}
\chi\wedge\omega+\partial\psi+\bar\partial\bar\psi\geq\delta\omega^2.
\end{equation}
In particular, we have \begin{equation}\label{eq:imaginary-omega-1}
\int_M\chi\wedge\omega^2>0,
\end{equation}
and for any irreducible surface $Y$,
\begin{equation}\label{eq:imaginary-omega-Y-1}
\int_Y\chi\wedge\omega\geq \delta\int_Y\omega^2.
\end{equation}

Set  $\eta=a\omega$ for a constant $a>0$ such that $3\int_M\chi^2\wedge\eta>\int_M\eta^3$. 
Such a choice is possible because of the assumption $\int_M\chi^2\wedge\omega>0$. 
Then
\begin{equation}\label{eq:imaginary-three}
\int_M\Ima(\chi+\ii\eta)^3
=3\int_M\chi^2\wedge\eta-\int_M\eta^3
>0.
\end{equation}
The inequality~\eqref{eq:imaginary-omega-1} gives \begin{equation}\label{eq:imaginary-two}
\int_M\Ima(\chi+\ii\eta)^2\wedge\eta
=2a^2\int_M\chi\wedge\omega^2>0,
\end{equation}
and the inequality~\eqref{eq:imaginary-omega-Y-1} implies for any irreducible surface $Y$,
\begin{equation}\label{eq:imaginary-omega-Y-2}
\int_Y\chi\wedge\eta\geq\frac{\delta}{a}\int_Y\eta^2.
\end{equation}
For $0\leq p\leq3$  and any constant $\theta\in(0,\pi)$,  we consider
\[
F_{\theta,p}(\chi)=\Rea(\chi+\ii\eta)^p-\cot\theta\,\Ima(\chi+\ii\eta)^p,
\quad F_{\theta,0}=1.
\]
Since $\cot\theta\to-\infty$ as $\theta\to\pi^{-}$, and $\eta$ is a K\"ahler form, equations~\eqref{eq:imaginary-three}--\eqref{eq:imaginary-two} and compactness allow us to choose a fixed constant $\theta\in(\pi/2,\pi)$ sufficiently close to $\pi$ such that
\begin{equation}\label{eq:F-positive-integrals}
\int_MF_{\theta,3}(\chi)>0,
\qquad
\int_MF_{\theta,2}(\chi)\wedge\eta>0,
\end{equation}
and
\begin{equation}\label{eq:F-pointwise}
\chi-\cot\theta\,\eta>\eta,
\qquad
\Rea(\chi+\ii\eta)^2-\frac{2\delta}{a}\cot\theta\,\eta^2>\eta^2.
\end{equation}

The first inequality in~\eqref{eq:F-pointwise} gives

\begin{equation}
    \int_MF_{\theta,1}(\chi)\wedge\eta^2>\int_M\eta^3.
\end{equation}

For any irreducible surface $Y$, using inequality~\eqref{eq:imaginary-omega-Y-2} and inequalities in ~\eqref{eq:F-pointwise}, we obtain
\begin{equation}\label{eq:F-surface}
\begin{aligned}
\int_Y F_{\theta,2}(\chi)
&= \int_Y \Rea(\chi+\ii\eta)^2
   -2\cot\theta\int_Y\chi\wedge\eta \\
&\ge \int_Y   \left(\Rea(\chi+\ii\eta)^2
   -\frac{2\delta}{a}\cot\theta\,\eta^2\right)
    > \int_Y \eta^2,
\end{aligned}
\end{equation}
and 
\begin{equation}
\int_YF_{\theta,1}(\chi)\wedge \eta>\int_Y\eta^2.
\end{equation}
For any irreducible curve $C$, the first inequality in~\eqref{eq:F-pointwise} gives
\begin{equation}\label{eq:F-curve}
\int_CF_{\theta,1}(\chi)>\int_C\eta.
\end{equation}
Consider the family $\chi_t=\chi+t\eta$, $t\geq0$. It is a test family in the sense of Chen~\cite{Chen}: $\chi_0=\chi$, the difference $\chi_s-\chi_t$ is positive for $s>t$, and $\chi_t-\cot(\theta/3)\eta>0$ for sufficiently large $t$. For any $p$-dimensional irreducible analytic subvariety $V$,
\begin{equation}\label{eq:F-binomial}
\int_VF_{\theta,p}(\chi_t)
=\sum_{k=0}^p\binom{p}{k}t^{p-k}\int_VF_{\theta,k}(\chi)\wedge\eta^{p-k}.
\end{equation}
All coefficients on the right-hand side are nonnegative by~\eqref{eq:F-positive-integrals}--\eqref{eq:F-curve}. In particular,
\begin{equation}\label{eq:chen-uniform}
\int_VF_{\theta,p}(\chi_t)\geq\frac{3-p}{2}\int_V\eta^p,
\qquad 1\leq p\leq3.
\end{equation}
Set
\[
c_\theta=\frac{\int_MF_{\theta,3}(\chi)}{\int_M\eta^3}>0,
\qquad
\Theta=\frac{\pi+\theta}{2}\in(\theta,\pi).
\]
Chen's existence theorem~\cite[Proposition~5.2]{Chen}, applied to~\eqref{eq:chen-uniform}, yields a smooth function $v$ satisfying
\begin{equation}\label{eq:chen-solution}
F_{\theta,3}(\chi_{v})=c_\theta\eta^3,
\qquad
\chi_{v}\in\Gamma_{\eta,\theta,\Theta}.
\end{equation}
Lemma~\ref{lem:phase-cone} gives
\[
\chi_{v}\in\Gamma^2_\eta=\Gamma^2_\omega.
\]
 Applying the solvability theorem for the complex Hessian equation~\cite{Szekelyhidi} to the background form $\chi_{v}$ and using
\[
\int_M\chi_{v}^2\wedge\omega
=\int_M\chi^2\wedge\omega
=\int_M f\omega^3,
\]
we obtain a smooth function $h$ such that
\[
(\chi_{v}+\ii\partial\bar\partial h)^2\wedge\omega=f\omega^3,
\qquad
\chi_{v}+\ii\partial\bar\partial h\in\Gamma^2_\omega.
\]
Taking $u=v+h$ proves statement~(3).
\end{proof}

\section{Consequences for numerical conjectures}
In this section, we first  solve Sz\'ekelyhidi's conjecture in dimension three by Theorem \ref{thm:main-threefold} and the Hodge index inequality in Lemma \ref{lem:hodge-index-inequality}. Similar arguments can also be used to solve Murakami's conjecture in dimension three. 
\subsection{Sz\'ekelyhidi's conjecture in dimension three}

\begin{theorem}\label{thm:szekelyhidi-dim-three}
	Let $(M,\omega)$ be a compact K\"ahler threefold, and let $\chi\in\Gamma^2_\omega$ be closed. Assume that
	\[
	\int_M\chi^2\wedge\omega=c\int_M\chi\wedge\omega^2,
	\qquad c>0.
	\]
	The following statements are equivalent.
	\begin{enumerate}[label=(\arabic*),leftmargin=2.4em]
		\item For any irreducible analytic surface $Y$,
		\[
		\int_Y(2\chi\wedge\omega-c\omega^2)>0.
		\]
		\item There exists a smooth function $u$ such that
		\[
		\chi_u^2\wedge\omega=c\chi_u\wedge\omega^2,
		\qquad
		\chi_u\in\Gamma^2_\omega.
		\]
	\end{enumerate}
\end{theorem}

\begin{proof}
	Suppose statement~(2) holds. In local coordinates diagonalizing $\chi_u$ with eigenvalues $\lambda_1,\lambda_2,\lambda_3$, the quotient equation is $\sigma_2(\lambda)=c\sigma_1(\lambda)$. For each $i$, let $\{i,j,k\}=\{1,2,3\}$. Then
	\[
	(\lambda_j+\lambda_k-c)\sigma_1(\lambda)
	=(\lambda_j+\lambda_k)\sigma_1(\lambda)-\sigma_2(\lambda)
	=\lambda_j^2+\lambda_j\lambda_k+\lambda_k^2>0.
	\]
	Since $\sigma_1(\lambda)>0$, we obtain $2\chi_u\wedge\omega-c\omega^2>0$, which implies statement~(1).
	
	Assume statement~(1), and set $
	\alpha=\chi-\frac{c}{2}\omega.$
	The compatibility condition gives
	\[
	\int_M\alpha^2\wedge\omega=\frac{c^2}{4}\int_M\omega^3.
	\]
	By Lemma~\ref{lem:hodge-index-inequality},
	\begin{equation}\label{eq:alpha-hodge}
		\left|\int_M\alpha\wedge\omega^2\right|
		\geq\frac{c}{2}\int_M\omega^3.
	\end{equation}
	Since $\chi\in\Gamma^2_\omega$, one has $\int_M\chi\wedge\omega^2>0$. The negative branch of inequality~\eqref{eq:alpha-hodge} would imply $\int_M\chi\wedge\omega^2\leq0$, so necessarily
	\[
	\int_M\alpha\wedge\omega^2\geq\frac{c}{2}\int_M\omega^3>0.
	\]
	Furthermore, statement~(1) gives
	\[
	\int_Y\alpha\wedge\omega>0
	\]
	for any irreducible surface $Y$.  Thus the numerical conditions in Theorem~\ref{thm:main-threefold} are satisfied for $\alpha$. Applying Theorem~\ref{thm:main-threefold}, we  obtain  a smooth function $u$ solving
	\[
	\alpha_{u}^2\wedge\omega=\frac{c^2}{4}\omega^3,
	\qquad
	\alpha_{u}\in\Gamma^2_\omega.
	\]
	Then $\chi_u=\alpha_u+\frac{c}{2}\omega\in\Gamma^2_\omega$ and 
	\[
	\chi_u^2\wedge\omega=c\chi_u\wedge\omega^2,
	\]
	which proves statement~(2).
\end{proof}

\subsection{Murakami's conjecture in dimension three}

\begin{theorem}\label{thm:murakami-dim-three}
Let $(M,\omega)$ be a compact K\"ahler threefold, and let $\chi$ be a real closed $(1,1)$-form satisfying
\[
\int_M\chi^2\wedge\omega=c\int_M\omega^3,
\qquad c>0.
\]
The following statements are equivalent.
\begin{enumerate}[label=(\arabic*),leftmargin=2.4em]
\item For any $t\geq0$,
\[
\int_M(\chi+t\omega)^2\wedge\omega>0,
\]
and, for any irreducible analytic surface $Y$,
\[
\int_Y\chi\wedge\omega>0.
\]
\item There exists a smooth function $u$ such that
\[
\chi_u^2\wedge\omega=c\omega^3,
\qquad
\chi_u\in\Gamma^2_\omega.
\]
\end{enumerate}
\end{theorem}

\begin{proof}
If statement~(2) holds, then $\chi_u+t\omega\in\Gamma^2_\omega$ for any $t\geq0$, and the numerical inequalities in statement~(1) follow.

Assume statement~(1), and define
\[
I=\{t\geq0:[\chi+t\omega]\text{ contains a smooth representative in }\Gamma^2_\omega\}.
\]
For sufficiently large $t$, $\chi+t\omega\in\Gamma^2_\omega$, so $I$ is nonempty. The openness of $\Gamma^2_\omega$ shows that $I$ is open in $[0,+\infty)$, and the monotonicity of the cone $\Gamma^2_\omega$ gives
\[
s\in I,\quad t>s\quad\Longrightarrow\quad t\in I.
\]
Let $t_0=\inf I$. Then $(t_0,\infty)\subset I$. Put
\[
A(t)=\int_M(\chi+t\omega)^2\wedge\omega,
\qquad
B(t)=\int_M(\chi+t\omega)\wedge\omega^2,
\qquad
V=\int_M\omega^3.
\]
By assumption, $A(t)>0$ for all $t\geq0$, and $A(t)\to\infty$ as $t\to\infty$. Hence there exists $C>0$ such that $A(t)\geq C$ on $[0,\infty)$. For $t>t_0$, admissibility gives $B(t)>0$. Lemma~\ref{lem:hodge-index-inequality} yields
\[
B(t)^2\geq A(t)V\geq CV,
\]
so $B(t)\geq\sqrt{CV}$. Passing to the limit gives $B(t_0)>0$. Moreover,
\[
\int_Y(\chi+t_0\omega)\wedge\omega>0
\]
for any irreducible surface $Y$. Thus the numerical conditions in Theorem~\ref{thm:main-threefold} are satisfied for $\chi+t_0\omega$. Applying Theorem~\ref{thm:main-threefold}, we obtain a $2$-admissible representative in $[\chi+t_0\omega]$. Thus $t_0\in I$. If $t_0>0$, openness would produce points of $I$ below $t_0$, a contradiction. Hence $t_0=0$, and $[\chi]$ contains a $2$-admissible representative. The standard existence theorem for the complex Hessian equation~\cite{Szekelyhidi} then gives statement~(2).
\end{proof}

\section{The boundary case}

We first prove the boundary form of the Gauduchon criterion. The following elementary linear-algebra lemma removes the only nonuniform step in the choice of the exceptional coefficients.

\begin{lemma}\label{lem:matrix}
Let $A$ be a real symmetric $m\times m$ matrix such that
\[
x^{T}Ax<0
\]
for any nonzero vector $x\in\mathbb{R}_{\geq0}^m$. Then there exists $a\in\mathbb{R}_{>0}^m$ such that $Aa<0$ componentwise.
\end{lemma}

\begin{proof}
Set $B=-A$ and minimize $x^TBx$ on the simplex
\[
\Delta=\left\{x\in\mathbb{R}_{\geq0}^m:\sum_i x_i=1\right\}.
\]
The minimum $\mu$ is positive. Suppose that $x_0$ is a minimizer. Let $e_j\in\Delta$ be the $j$-th standard basis vector.  For each $j$, choose a path $$x_j(t)=(1-t)x_0+te_j,\quad t\in[0,1]$$ in $\Delta$ connecting $x_0$ and $e_j$. Define a smooth function $f_j:[0,1]\to\mathbb{R}$ by
\[
f_j(t)=x_j(t)^TBx_j(t).
\]
Then $t=0$  is a minimizer of $f_j(t)$, hence
\begin{equation}\label{eq:directional-derivative}
\left.\frac{d}{dt}\right|_{t=0}f_j(t)=2(e_j-x_0)^T Bx_0\geq0.
\end{equation}
Inequality~\eqref{eq:directional-derivative} gives
\[
(Bx_0)_j\geq x_0^TBx_0=\mu
\]
for every $j$. Thus $Bx_0$ is strictly positive. Replacing $x_0$ by $x_0+\varepsilon(1,\ldots,1)$ for sufficiently small $\varepsilon>0$ preserves this strict positivity and makes all coordinates positive. The resulting vector $a$ satisfies $Aa<0$.
\end{proof}

{\begin{theorem}
			\label{prop:boundary-detailed}
			Under the hypotheses of Theorem~\ref{prop:boundary-gauduchon}, there exist finitely many
			exceptional prime divisors $D_1,\ldots,D_m$ and positive numbers
			$a_1,\ldots,a_m$ such that, for
			\[
			\chi'=\chi-\sum_{i=1}^m a_i\theta_i,
			\]
			where $\theta_i$ is a smooth representative of $\{D_i\}$, we have
			\[
			[\chi'\wedge\omega^{n-2}]_A\in\mathcal G
			\]
			and
			\[
			\int_M(\chi')^2\wedge\omega^{n-2}>0 .
			\]
		\end{theorem}
		
		\begin{proof}
			Recall 
			\[
			q(\alpha,\beta)=\int_M\alpha\wedge\beta\wedge\omega^{n-2},\qquad \alpha,\beta\in H^{1,1}(M,\mathbb{R}).
			\]
			By the hypotheses of Theorem~\ref{prop:boundary-gauduchon}, we can choose $\delta>0$ sufficiently small such that
			\[
			q(\chi-\delta\omega,\chi-\delta\omega)>0,
			\qquad
			q(\chi-\delta\omega,\omega)>0 .
			\]
			Set
			\[
			\widetilde\chi=\chi-\delta\omega
			\]
			and define
			\[
			S=\{D\subset M:\ D\text{ is a prime divisor and }
			q(\widetilde\chi,\{D\})\le0\}.
			\]
			
			We first show that every finite collection of divisors in $S$ is an
			exceptional family. Let
			\[
			E=\sum_{i=1}^r x_i\{D_i\},
			\qquad x_i\ge0,
			\]
			be a nonzero element in the convex cone generated by the classes of a finite subset of $S$.
			Since $E$ is represented by a nonzero effective divisor, at least one $x_i>0$, hence
			\[
			q(\omega,E)=\int_E\omega^{n-1}>0 .
			\]
			By the definition of $S$,
			\[
			q(\widetilde\chi,E)\le0 .
			\]
			Define
			\[
			t=-\frac{q(\widetilde\chi,E)}{q(\omega,E)}\ge0
			\]
			and set
			$
			\alpha=\widetilde\chi+t\omega .
			$
			Then
			\[
			q(\alpha,E)=0 .
			\]
			Furthermore,
			\[
			\begin{aligned}
				q(\alpha,\alpha)
				=
				q(\widetilde\chi,\widetilde\chi)
				+2tq(\widetilde\chi,\omega)
				+t^2q(\omega,\omega)
				>0 .
			\end{aligned}
			\]
			Hence, by the Hodge index theorem, the intersection form $q$ is
			negative definite on $\alpha^\perp$. Since $E\neq0$ and
			$E\in\alpha^\perp$, we obtain
			\[
			q(E,E)<0 .
			\]
			
			Thus by Lemma~\ref{lem:modified-nef-square}, no nonzero element in the convex cone generated by the classes of a finite subset of $S$ can belong to the modified nef cone. Hence every finite
			subfamily of $S$ is an exceptional family in the sense of Boucksom
			\cite{Boucksom}. Since the classes of an exceptional family are
			linearly independent, the cardinality of every such family is bounded
			by the Picard number. Consequently, $S$ itself is finite. Write
			\[
			S=\{D_1,\ldots,D_m\}.
			\]
			
			If $m=0$, then
			\[
			q(\chi,\{D\})>q(\widetilde\chi,\{D\})>0
			\]
			for any prime divisor $D$, and Theorem~\ref{thm:gauduchon-class} immediately gives the
			conclusion. Hence we assume $m\ge1$.
			
			Let
			\[
			A=(q(\{D_i\},\{D_j\}))_{1\le i,j\le m}.
			\]
			The argument above shows that
			\(
			x^TAx<0
			\)
			for any nonzero vector
			\(
			x\in\mathbb R_{\ge0}^m.
			\)
			By Lemma~\ref{lem:matrix}, there exists
			\(
			b=(b_1,\ldots,b_m)\in\mathbb R_{>0}^m
			\)
			such that
			\(
			Ab<0
			\)
			componentwise. Let
			\[
			E=\sum_{i=1}^m b_i\{D_i\}.
			\]
			Then
			\[
			q(E,\{D_j\})<0,
			\qquad 1\le j\le m .
			\]
			
			Since
			\(
			q(\chi,\{D_j\})\ge0
			\)
			by the assumption of Theorem~\ref{prop:boundary-gauduchon}, we have
			\[
			q(\chi-sE,\{D_j\})>0,\qquad 1\le j\le m
			\]
			for any sufficiently small $s>0$.
			
			Choose $s>0$ sufficiently small such that
			\(
			\delta[\omega]-sE
			\)
			is a K\"ahler class and
			\[
			q(\chi-sE,\chi-sE)>0,
			\qquad
			q(\chi-sE,\omega)>0 .
			\]
			For any prime divisor $D\notin S$, we have
			\[
			q(\widetilde\chi,\{D\})>0.
			\]
			Hence
			\[
			\begin{aligned}
				q(\chi-sE,\{D\})
				&=
				q(\widetilde\chi,\{D\})
				+
				q(\delta\omega-sE,\{D\})\\
				&>0 ,
			\end{aligned}
			\]
			because $\delta[\omega]-sE$ is a K\"ahler class and $D$ is an effective
			divisor.
			
			Therefore,
			\[
			q(\chi-sE,\{D\})>0
			\]
			for any prime divisor $D$.

 Set	$a_i=sb_i,\;1\leq i\leq m.$ For smooth representatives $\theta_i$ of $\{D_i\}$,
			we have
			$$
		[\chi-\sum_{i=1}^m a_i\theta_i]=[\chi]-sE.	
			$$

           Set $\chi'=\chi-\sum_{i=1}^m a_i\theta_i.$  Then
			\[
			q(\chi',\chi')=q(\chi-sE,\chi-sE)>0,
			\]
and \[q(\chi',\{D\})=q(\chi-sE,\{D\})>0\]
			for any prime divisor $D$.
            
 Applying Theorem~\ref{thm:gauduchon-class} to
			\(
			\chi'
			\)
			gives
			\[
			[\chi'\wedge\omega^{n-2}]_A\in\mathcal G .
			\]
So the conclusion follows.
\end{proof}}

\begin{proof}[Proof of Theorem~\ref{prop:boundary-gauduchon}]
This is Theorem~\ref{prop:boundary-detailed}, expressed cohomologically.
\end{proof}

We now give an application of the boundary criterion in dimension three. Following~\cite{PangSunWangZhou}, a smooth real $(1,1)$-form $\chi$ is called $(\omega,2)$-big if there exist a $(\chi,\omega,2)$-subharmonic quasi-psh function $\rho$ and a smooth form $\gamma\in\Gamma^2_\omega$ such that
\[
\chi+\ii\partial\bar\partial\rho\geq\gamma
\]
as currents.

\begin{proposition}\label{thm:big-boundary}
Let $(M,\omega)$ be a compact K\"ahler threefold, and let $\chi$ be a real closed $(1,1)$-form. Suppose  the hypotheses in the first assertion of Corollary~\ref{thm:boundary-hessian} hold, namely, \[
\int_M\chi^2\wedge\omega>0,\qquad \int_M\chi\wedge\omega^2>0,
\]
and that, for any irreducible analytic surface $Y\subset M$,
\[
\int_Y\chi\wedge\omega\geq0.
\] Then  $\chi$ is $(\omega,2)$-big.
\end{proposition}

\begin{proof}
By Theorem~\ref{prop:boundary-detailed}, there exist prime divisors $D_i$, positive numbers $a_i$, and smooth curvature forms $\theta_i$ representing the classes $\{D_i\}$ such that, for
\[
\beta=\chi-\sum_i a_i\theta_i,
\]
the class $[\beta\wedge\omega]_{\A}$ is Gauduchon and
\[
\int_M\beta^2\wedge\omega>0.
\]
Then statement~(2) of Theorem~\ref{thm:main-threefold} is satisfied for $\beta$. Applying Theorem~\ref{thm:main-threefold}, we  obtain a smooth function $v$ such that
\[
\gamma=\beta+\ii\partial\bar\partial v\in\Gamma^2_\omega.
\]
Let $h_i$ be Hermitian metrics on $\mathcal{O}(D_i)$ with curvature $\theta_i$, and let $s_i$ be the canonical sections vanishing on $D_i$. By the Lelong--Poincar\'e formula~\cite{DemaillyBook},
\[
\ii\partial\bar\partial\log|s_i|_{h_i}^2=[D_i]-\theta_i
\]
as currents. Define
\[
\rho=v+\sum_i a_i\log|s_i|_{h_i}^2.
\]
Then, in the sense of currents,
\[
\chi+\ii\partial\bar\partial\rho
=\gamma+\sum_i a_i[D_i]\geq\gamma.
\]
By the monotonicity of the $2$-subharmonic cone, $\rho$ is $(\chi,\omega,2)$-subharmonic, and $\chi$ is $(\omega,2)$-big.
\end{proof}

\begin{proof}[Proof of Corollary \ref{thm:boundary-hessian}]
The first assertion is Proposition~\ref{thm:big-boundary}. Suppose, in addition, that $\chi\in\overline{\Gamma}^2_{\omega}$ and
\[
\int_M\chi^2\wedge\omega>0.
\]
The assumption \(\chi\in\overline\Gamma_\omega^2\) implies that $\chi\wedge\omega\geq0$ as a $(2,2)$-form, so
\[
\int_Y\chi\wedge\omega\geq0
\]
for any irreducible surface $Y$. It also gives $\int_M\chi\wedge\omega^2\geq0$, while Lemma~\ref{lem:hodge-index-inequality} and $\int_M\chi^2\wedge\omega>0$ imply that this integral is nonzero. Hence it is strictly positive. The first assertion therefore shows that $\chi$ is $(\omega,2)$-big.   The existence and uniqueness theorem on degenerate complex Hessian equations for $(\omega,2)$-big and 2-semipositive forms~\cite[Theorem~11.1]{PangSunWangZhou} yields a unique function
$
u\in\operatorname{SH}_2(M,\chi,\omega)\cap L^\infty(M)$
satisfying
\[
(\chi+\ii\partial\bar\partial u)^2\wedge\omega=e^{\lambda u}f\omega^3
\]
in the potential sense, and
\[
\operatorname{osc}_M u \leq C,
\]
where $C$ is a constant depending on
$\chi,\omega,\lambda,p,M,\|f\|_p$.
\end{proof}

\end{document}